\documentclass[11pt]{amsart}
\usepackage{amscd,amssymb,verbatim}
\theoremstyle{plain}
\newtheorem{Thm}{Theorem}[section]
\newtheorem{Cor}[Thm]{Corollary}
\newtheorem{Lem}[Thm]{Lemma}

\newtheorem{remark}{Remark}

\numberwithin{equation}{section}

\begin{document}
\title[A new isomorphic \(\ell_1\) predual]{A new isomorphic \(\ell_1\) predual 
not isomorphic to 
a complemented subspace of a \(C(K)\) space}
\author{I. Gasparis}
\address{Department of Mathematics \\
Aristotle University of Thessaloniki \\
Thessaloniki, 54124 \\
Greece}
\email{ioagaspa@math.auth.gr}
\keywords{\(\ell_1\) predual, projection, Bourgain-Delbaen construction}
\subjclass{Primary: 46B03; Secondary: 47B38}
\begin{abstract}
We construct a subspace \(X\) of \(C(\omega^\omega)\) with dual isomorphic to
\(\ell_1\) and such that neither \(X\) embeds into \(c_0\), nor 
\(C(\omega^\omega)\) embeds into \(X\). As a consequence, \(X\) is not isomorphic
to a complemented subspace of a \(C(K)\) space. 
\end{abstract}
\maketitle
\section{Introduction}
Y. Benyamini and J. Lindenstrauss \cite{BL} constructed an isometric
\(\ell_1\) predual \(E\) which is not isomorphic to a complemented
subspace of a \(C(K)\) space.
D. Alspach and Y. Benyamini \cite{AB} showed, with a different proof, that a variation of \(E\)
had the same property. Alspach's quotient of \(C(\omega^\omega)\) which does not embed
into any \(C(\alpha)\), \(\alpha < \omega_1\) \cite{A2}, is an isometric \(\ell_1\) predual which
contains a complemented copy of \(E\) and so it is also not isomorphic to a complemented
subspace of a \(C(K)\) space.
We remark that \(E\) is isometric to a subspace of \(C(\omega^\omega)\) and that all
aforementioned examples contain a copy of \(C(\omega^\omega)\). 

The preceding examples are related to the problem of the isomorphic classification
of the complemented subspaces of \(C(K)\) spaces (\cite{R1}, \cite{R2}). 
By combining results from \cite{A1}, \cite{Be}, \cite{JZ} one obtains the following 
structural result for complemented subspaces of \(C(K)\) spaces with separable dual:  
\begin{Thm} \label{T}
Let \(Y\) be a complemented subspace of \(C(K)\) with \(Y^*\) separable. Then either \(Y\)
is isomorphic to \(c_0\), or \(C(\omega^\omega)\) embeds into \(Y\).
\end{Thm}
The proofs that the preceding examples of \(\ell_1\) preduals are not isomorphic to
complemented subspaces of \(C(K)\) spaces are rather hard and part of this difficulty
lies on the fact that they all contain copies of \(C(\omega^\omega)\). As a consequence, the
preceding structural result can not be used to distinguish between these \(\ell_1\) preduals
and the complemented subspaces of \(C(K)\) spaces. 

It follows directly from Theorem \ref{T}, that any isomorphic \(\ell_1\) predual not isomorphic to
\(c_0\) and not containing \(C(\omega^\omega)\) isomorphically, is not isomorphic to a 
complemented subspace of a \(C(K)\) space. The purpose of this article, as our next result
indicates, is to provide such an example of an \(\ell_1\) predual.
\begin{Thm} \label{mT}
There exists an isomorphic \(\ell_1\) predual \(X\) with a normalized, shrinking
basis \((e_n)\) satisfying the following properties:
\begin{enumerate}
\item \(X\) is isometric to a subspace of \(C(\omega^\omega)\).
\item Every normalized weakly null sequence in \(X\) admits a subsequence which is either
equivalent to the \(c_0\) basis, or equivalent to a subsequence of the natural basis
of Schreier's space. In particular, \((e_n)\) satisfies the second alternative.
\end{enumerate}
\end{Thm}
Recall that Schreier's space \cite{Sch} is the completion of \(c_{00}\) under the norm
\[ \|x\| = \sup \{ \sum_{i \in F} |x(i)|: \, |F| \leq \min F \} \]
The unit vector basis of \(c_{00}\) then becomes the natural Schauder basis of Schreier's space.
\begin{Cor}
Neither \(X\) embeds into \(c_0\), nor \(C(\omega^\omega)\) embeds into \(X\).
Therefore, \(X\) is not isomorphic to a complemented subspace of a \(C(K)\) space.
\end{Cor}
\begin{proof}
Since \((e_n)\) has a a subsequence equivalent to a subsequence of the natural
basis of Schreier's space, and since none of these subsequences is equivalent
to the \(c_0\) basis, we infer that \(X\) is not isomorphic to a subspace of \(c_0\).

One consequence of Theorem \ref{mT} is that the only spreading models of \(X\)
are \(c_0\) and \(\ell_1\). On the other hand, it is not hard to see that 
\(C(\omega^\omega)\) admits spreading models which are not equivalent to the 
\(c_0\) or to the \(\ell_1\) bases.
In fact, it is well known (\cite{O}, \cite{AK}), that \(C(\omega^\omega)\)
is a space universal for all unconditional spreading models. It follows now
that \(C(\omega^\omega)\) is not isomorphic to a subspace of \(X\).

Finally, we deduce from the above and Theorem \ref{T} that 
\(X\) is not isomorphic to a complemented subspace of a \(C(K)\) space.
\end{proof}
\begin{remark}
It was pointed out to us by D. Alspach, that \(X\) is not even isomorphic to
a quotient of a \(C(K)\) space with \(K\) being countable and compact.
Indeed, let \(Q \colon C(K) \to X\) be a bounded linear surjection, where \(K\)
is countable. It is shown in \cite{AJO} that the Szlenk index \cite{Szl} of any subspace
of Schreier's space spanned by a subsequence of the natural basis equals the
Szlenk index of \(C(\omega^\omega)\), which in turn equals \(\omega^2\) \cite{Sam}.
It follows that the Szlenk index of \(X\) also equals \(\omega^2\) and therefore, as
\(K\) is countable, there is some \(\epsilon > 0\) such that the \(\epsilon\)-Szlenk index of
\(Q^* B_{X^*}\) exceeds \(\omega\). One deduces from the main result of \cite{A1}
that \(Q\) is an isomorphism on some subspace isomorphic to \(C(\omega^\omega)\),
contradicting Theorem \ref{mT}. Combining this result with that of \cite{Be} we conclude that \(X\)   
is not isomorphic to a complemented subspace of any \(C(K)\) space.
\end{remark}
The construction of \(X\) requires a new method for producing \(\mathcal{L}_\infty\)
spaces (\cite{LP}, \cite{LR}). More precisely we use a dual version
of the Bourgain-Delbaen method \cite{BD}. The latter has been instrumental to some
recent developments in Banach space theory, namely the solution of the scalar-plus-
compact problem \cite{AH} and the universality of the class of \(\ell_1\) 
preduals among spaces with separable dual \cite{FOS}. Theorem \ref{T1} gives sufficient
conditions on a norming subset of the dual ball of a Banach space with a basis, in order for
the space to be \(\mathcal{L}_\infty\).

We use standard Banach space facts and terminology, as may be found in \cite{LT}.
By a subspace of a Banach space we shall always mean an infinite-dimensional, closed subspace.
If \(\alpha < \omega_1\) is a countable ordinal, then \(C(\alpha)\) stands for
\(C(K)\) where \(K\) is the ordinal interval \([1, \alpha]\) endowed with the order topology. 
\section{A criterion for recognizing \(\mathcal{L}_\infty \) spaces}
{\bf Notation}. Let \(X\) be a Banach space with a normalized Schauder basis
\((e_n)\) and let \(D \subset B_{X^*}\) be a norming set for \(X\) 
(that is, \(\|x \| = \sup_{d^* \in D } |d^*x|\) for all \(x \in X\)) so that
\(D \subset < e_n^* : \, n \in \mathbb{N}> \setminus  \{0\}\). Assume that
\(e_n^* \in D\) for all \(n \in \mathbb{N}\) and that \(|d^*(e_n)| \leq 1\)
for all \(d^* \in D\) and all \(n \in \mathbb{N}\).

We also consider a sequence \(\Delta_1 < \Delta_2 < \cdots < \Delta_n < \cdots\)
of successive finite intervals of \(\mathbb{N}\) whose union is \(\mathbb{N}\).
Assume that \(|\mathrm{ supp } \, d^* \cap \Delta_n | \leq 1\)
for all \(d^* \in D\) and all \(n \in \mathbb{N}\).

We set \(D_n = \{ d^* \in D : \, \max \mathrm{ supp } \, d^* \in \Delta_n \}\)
for all \(n \in \mathbb{N}\). Thus, \(D = \cup_n D_n\). 

We finally let \(P_n\) denote the basis projection onto
\([e_i : \, i \in \cup_{k=1}^n \Delta_k ]\).
\begin{Thm} \label{T1}
Let \(X\), \((e_n)\), \(D\) and \((\Delta_n)\) be as above.
Let \( 0 < b < 1/4\). Assume that the following properties hold for all \(n \geq 3\):
\begin{enumerate}
\item \(D_n = \{\gamma_i^* : \, i \in \Delta_n \} \cup \{e_i^* : \, i \in \Delta_n \}\),
where
for each \(i \in \Delta_n\),
\( | \mathrm{ supp } \, \gamma_i^* | > 1\), \(\max \mathrm{ supp } \, \gamma_i^* = i\)
and \(\gamma_i^*e_i = 1\).
\item Each \(d^* \in D_n\) admits a representation of the form
\[ d^* = \rho_1 d_1^* + \rho_2( d_2^* | \cup_{j=k+1}^l \Delta_j )+ e_i^*\]
where, \(d_1^* \in D_k\) and \(d_2^* \in D_l\) for some \(1 \leq k < l \leq n-1\),
\(i \in \Delta_n\) and \(|\rho_1 | \leq 1\), \(|\rho_2 | \leq b\).
\end{enumerate}
Then, \(X\) is an \(\mathcal{L}_{\infty}\) space. 
Moreover, letting 
\(b_i^* = e_i^*\) for \(i \in \Delta_1 \cup \Delta_2\), while if \(n \geq 3\)
and \(i \in \Delta_n\),
\[b_i^* = (1/2) \gamma_i^* P_{n-1} + e_i^*\] 
then \((b_i^*)\) is equivalent to the \(\ell_1\) basis and \([(b_i^*)]=[(e_i^*)]\).
Hence, if \((e_n)\) is shrinking then \(X^*\) is isomorphic to \(\ell_1\).
\end{Thm}
\begin{proof}
For each \(n \geq 2\) we define linear maps
\(T_n \colon \ell_\infty(\cup_{k=1}^n \Delta_k) \to [e_i : \, i \in \cup_{k=1}^n \Delta_k]\)
as follows: 
\[T_2 x = \sum_{i=1}^m x(i) e_i\]
where \(m = \max \Delta_2\) and inductively,
\[T_{n+1} x = T_n \pi_n x + \sum_{i \in \Delta_{n+1} } 
[ x(i) - (1/2)\gamma_i^* T_n \pi_n x] e_i \]
where \(\pi_n \colon \ell_\infty \to \ell_\infty\) is the restriction operator to the
first \(\cup_{k=1}^n \Delta_k\) coordinates and \(\gamma_i^*\) is the unique element of \(D\)
whose support contains at least two points and \(i\) is the maximum of this support.
It is clear that
\[ P_m T_n x = T_m \pi_m x\]
whenever \(m \leq n\) and \(x \in \ell_\infty(\cup_{k=1}^n \Delta_k)\).

It will suffice showing that there exist absolute constants \(0 < A < B\) so that
\[A \|x\|_\infty \leq \|T_n x \| \leq B \|x\|_\infty\]
for all \(x \in \ell_\infty(\cup_{k=1}^n \Delta_k)\) and \(n \geq 2\).   

Let \(\rho = (1/2)[ 1 + 3b/(1-b)]\). Then \( 1/2 < \rho < 1\) as \( 0 < b < 1/4\).
We choose \( \lambda  > 0 \) such that
\(\|T_2 \| \leq 1 + \lambda /2\), \(\|(I - P_1) T_2 \| \leq 1 + 3\lambda /2\)
and \(\lambda > (1- \rho)^{-1} (1- b)^{-1}\).

We first show by induction on \(n \geq 2\) that for all 
\(x \in \ell_\infty(\cup_{k=1}^n \Delta_k)\), \(\|x\|_\infty =1\), there exist
\(d^* \in \cup_{k=1}^n D_k\) and an initial interval \(I\) of \(\mathbb{N}\)
so that
\[|(d^* | I)(T_n x) | \geq 1/2\]
The assertion trivially holds for \(n=2\), as \(e_j^* \in D_k\) whenever
\(j \in \Delta_k\) for some \(k \leq 2\). Assuming the assertion true for some \(n \geq 2\),
let \(x \in \ell_\infty(\cup_{k=1}^{n +1} \Delta_k)\) with \(\|x\|_\infty =1\) and choose
\(k \leq n+1 \) and \(i_0 \in \Delta_k\) so that \(|x(i_0)| = 1\).
In case \(k \leq n\), the assertion follows from the induction hypothesis as
\(P_n T_{n+1}x = T_n \pi_n x\). If \(i_0 \in \Delta_{n+1}\), we distinguish between
two cases. The first one occurs when \(|\gamma_{i_0}^* T_n \pi_n x | \geq 1\).
In this case, \(d^* = \gamma_{i_0}^* \) and \( I = \cup_{k=1}^n \Delta_k \) do the job.
The second case occurs when  \(|\gamma_{i_0}^* T_n \pi_n x | < 1\). Now,
\(d^* = \gamma_{i_0}^* \) and \( I = \cup_{k=1}^{n + 1} \Delta_k \) fulfill the assertion.
This is so since \(\mathrm{ supp } \gamma_{i_0}^* \cap \Delta_{n+1} = \{ i_0\} \).

This completes the inductive step and hence the assertion holds for all \(n \geq 2\).
It follows now that \(\|T_n x \| \geq (2C)^{-1} \|x\|_{\infty}\) for all
\(x \in \ell_\infty(\cup_{k=1}^n \Delta_k)\) and \(n \geq 2\), where \(C\)
stands for the basis constant of \((e_n)\). This yields the required lower estimate.

To obtain the upper estimate,
we show by induction on \(n \geq 2\) that the following properties hold:
\begin{enumerate}
\item \(\|d^* T_n \| \leq 1 + \lambda/2\), for all \(d^* \in \cup_{k=1}^n D_k\).
\item \(\|d^* (I - P_m) T_n \| \leq 1 + 3\lambda/2\), for all \(d^* \in \cup_{k=1}^n D_k\)
and \(m \leq n\).
\item For every \(d^* \in D\) and all \(m \leq n\) there exists \(l > 0\) so that
\[ \|d^* ( I - P_m ) T_n \| \leq (1 + 3 \lambda / 2) \sum_{k=0}^l b^k \]
\item \(\|T_n \| \leq \lambda\).
\end{enumerate}
Property \((4)\) above, clearly implies the required upper estimate.
We prove, by induction on \(n \geq 2\), that all four properties, above, hold.
By the choices made initially and since \(\lambda > 2\), this is clear when \(n =2 \).
For the inductive step we assume that all four properties hold for \(k \leq n\) and
then prove they are also valid for \(n +1\). 

We first show that \((1)\) holds for \(n+1\). Let \(d^* \in D_k\) for some \(k \leq n+1\).
In case \(k \leq n\) then \(d^* T_{n+1} x = d^* T_k \pi_k x\) for all
\(x \in \ell_\infty(\cup_{j=1}^{n + 1} \Delta_j)\) with \(\|x\|_\infty \leq 1\)
and the assertion follows from
the induction hypothesis. We next assume that \(k = n+1\). Let
\(x \in \ell_\infty(\cup_{j=1}^{n + 1} \Delta_j)\) with \(\|x\|_\infty \leq 1\).
We distinguish between two cases.
The first case occurs when \(d^* = e_i^*\) for some \(i \in \Delta_{n+1}\).
Then, 
\[|d^*T_{n+1}x| = |x(i) -  (1/2) \gamma_i^*T_n \pi_n x | \leq 1 + \lambda / 2\]
by \((4)\) of the induction hypothesis applied on \(T_n\).
In the second case, \(d^* = \gamma_i^*\) for some \(i \in \Delta_{n+1}\). Now,
\[|d^*T_{n+1}x |= |x(i) + (1/2) \gamma_i^*T_n \pi_n x | \leq 1 + \lambda /2\]
and the assertion follows again from \((4)\) of the induction hypothesis applied on \(T_n\).
This completes the inductive step for \((1)\).

We next establish \((2)\) for \(n+1\). 
Let \(d^* \in D_k\) for some \(k \leq n+1\). If \(k \leq m\), the assertion is trivial.
So we assume that \(m < k\). 
If \(k \leq n\), then
\[ d^*(I-P_m)T_{n+1}x = d^*(I-P_m) T_k \pi_k x\]
for all \(x \in \ell_\infty(\cup_{j=1}^{n + 1} \Delta_j)\) with \(\|x\|_\infty \leq 1\)
and the assertion follows from \((2)\)
of the induction hypothesis applied on \(T_k\).

When  \(k = n+1\), let
\(x \in \ell_\infty(\cup_{j=1}^{n + 1} \Delta_j)\) with \(\|x\|_\infty \leq 1\).
Again, there are two possibilities.
If \(d^* = e_i^*\) for some \(i \in \Delta_{n+1}\), then
\[|d^*(I - P_m )T_{n+1}x| = |x(i) - (1/2) \gamma_i^*T_n \pi_n x| \leq 1 + \lambda/2\]
by \((4)\) of the induction hypothesis applied on \(T_n\).

The other possibility is to have \(d^* = \gamma_i^*\) for some \(i \in \Delta_{n+1}\).
We once again have that     
\[d^*T_{n+1}x = x(i) + (1/2) \gamma_i^*T_n \pi_n x\]
while, \(d^*P_m T_{n+1}x = \gamma_i^* T_m \pi_m x \) whence,
\[|d^*(I - P_m )T_{n+1} x| = |x(i) + (1/2) \gamma_i^*T_n \pi_n x - \gamma_i^* T_m \pi_m x | 
\leq 1 + \frac{\lambda}{2} + \lambda = 1 + 3 \lambda/2\]
by \((4)\) of the induction hypothesis applied on \(T_n\) and \(T_m\). So \((2)\) is proved
for \(n+1\).

We now show that \((3)\) holds for \(n+1\). Since \((2)\) has been already verified for
\(n+1\), it will suffice establishing \((3)\) for all \(d^* \in \cup_{k = n+1}^\infty D_k \).
To accomplish this we use induction on \(k \geq n+1\) to prove that every
\(d^* \in D_k\) satisfies \((3)\). When \(k = n+1\), this follows from \((2)\).
Assume the assertion holds for every \(d^* \in D_s\) and all \(n+1 \leq s < k\)
and let \(d^* \in D_k\). By the splitting property of the elements of \(D\) we may write
\[ d^* = \rho_1 d_1^* + \rho_2 d_2^* (I - P_s) + e_i^*\]
where, \(d_1^* \in D_s\) and \(d_2^* \in D_t\) for some \(1 \leq s < t \leq k-1\),
\(i \in \Delta_k\) and \(|\rho_1 | \leq 1\), \(|\rho_2 | \leq b\).
If \( s \geq n+1\) the assertion for \(d^*\) follows by the induction hypothesis
applied on \(d_1^*\). If \( s \leq m\) then the assertion for \(d^*\) follows by the
induction hypothesis applied on \(d_2^*\). We now assume that \(m < s < n+1\). Applying
\((2)\) on \(d_1^*\) we obtain that
\(\|d_1^*(I - P_m) T_{n+1} \| \leq 1 + 3 \lambda /2\).
On the other hand, the induction hypothesis applied on \(d_2^*\) yields some integer
\(l > 0\) so that
\(\|d_2^*(I - P_s) T_{n+1} \| \leq (1 + 3 \lambda /2) \sum_{j=0}^l b^j \).
Using the fact that \((I - P_s)(I - P_m) = I - P_s\), we infer that
\[\|d_2^*(I - P_s)(I -P_m) T_{n+1} \| \leq (1 + 3 \lambda /2) \sum_{j=0}^l b^j \]
whence
\[\|d^*(I - P_m) T_{n+1} \| \leq |\rho_1| (1 + 3 \lambda /2) + b (1 + 3 \lambda /2) \sum_{j=0}^l b^j \]
and so \(\|d^*(I - P_m) T_{n+1} \| \leq (1 + 3 \lambda /2) \sum_{j=0}^{l+1} b^j\),
proving \((3)\) for \(n+1\).

To show \((4)\) holds for \(n+1\), it suffices to prove by induction on \(k \)
that \(\|d^* T_{n+1} \| \leq \lambda\) for all \(d^* \in D_k\).
When \(k \leq n+1\) this follows from \((1)\) which has been already shown to hold for \(n+1\). 
Let \(k > n+1\) and 
assume the assertion holds for all \(d^* \in D_s\), \(s < k\). Suppose that \(d^* \in D_k\). 
We may write
\[ d^* = \rho_1 d_1^* + \rho_2 d_2^* (I - P_s) + e_i^*\]
where, \(d_1^* \in D_s\) and \(d_2^* \in D_t\) for some \(1 \leq s < t \leq k-1\),
\(i \in \Delta_k\) and \(|\rho_1 | \leq 1\), \(|\rho_2 | \leq b\).
In case \(s \geq n+1\) the assertion for \(d^*\) follows by the induction hypothesis
applied on \(d_1^*\). 

If \(s < n+1\), then \(\|d_1^* T_{n+1} \| \leq 1 + \lambda /2\) because of \((1)\).
On the other hand, \((3)\) implies that
\[\|d_2^* (I - P_s) T_{n+1}\| \leq (1 + 3 \lambda /2) \sum_{j=0}^l b^j\]
for some positive integer \(l\).
We conclude that
\begin{align}
\|d^* T_{n + 1} \| &\leq 1 + \lambda /2 + b (1 + 3 \lambda /2) \sum_{j=0}^l b^j \notag \\ 
&< 1 + \lambda /2 + (1 + 3 \lambda /2)b (1-b)^{-1} \notag \\
&= (1-b)^{-1} + \rho \lambda < \lambda \notag
\end{align}
This completes the inductive step.

Finally, for the moreover assertion, we show by induction on \(n \in \mathbb{N}\)
that \((b_i^*)_{i \in \cup_{k=1}^n \Delta_k}\)
acts biorthogonally onto \((T_n e_i)_{i \in \cup_{k=1}^n \Delta_k}\) 
(where, \(T_1 =  T_2 \pi_1\)).
Indeed, this is trivial when \(n \leq 2\). Assume the assertion holds for some \(n \geq 2\).
Let \(i\) and \(j\) belong to \(\cup_{k=1}^{n+1} \Delta_k\). In case 
\(i \in \cup_{k=1}^n \Delta_k\), then since \(b_i^*\) is supported by
\(\cup_{k=1}^n \Delta_k\), we obtain that 
\[b_i^* T_{n+1} e_j = b_i^* T_n \pi_n e_j = \delta_{i, j}\]
because of the induction hypothesis. If \(i \in \Delta_{n+1}\)
then \(b_i^* = (1/2) \gamma_i^* P_n + e_i^*\) and
\[b_i^* T_{n+1} e_j = (1/2) \gamma_i^* T_n \pi_n e_j + e_i^* e_j - (1/2) \gamma_i^* T_n \pi_n e_j =
\delta_{i, j} \]
completing the inductive step.
It follows now that for every \(n \in \mathbb{N}\) and all choices of scalars
\((a_i)_{i \in \cup_{k=1}^n \Delta_k} \), we have
\[\lambda^{-1} \sum_{i \in \cup_{k=1}^n \Delta_k} |a_i| \leq 
\| \sum_{i \in \cup_{k=1}^n \Delta_k} a_i b_i^* \| \leq (2^{-1} C + 1)
\sum_{i \in \cup_{k=1}^n \Delta_k} |a_i|\]
where \(C \) is the basis constant. This shows that \((b_n^*)\) is equivalent to the \(\ell_1\) basis.
Since for all \(n \in \mathbb{N}\),
the span of the \((b_i^*)\)'s, \({i \in \cup_{k=1}^n \Delta_k}\),
is contained in the span of the \((e_i^*)\)'s, \({i \in \cup_{k=1}^n \Delta_k}\),
we deduce from the linear independence of the \((b_i^*)\)'s,
that \([b_n^* : \, n \in \mathbb{N}] = [e_n^* : \, n \in \mathbb{N}]\).
\end{proof}
\begin{remark}
It is shown in \cite{LS} that if a \(\mathcal{L}_\infty\)
space has a separable dual, then this dual must be isomorphic to \(\ell_1\).
\end{remark}
\section{The construction of \(X\)} 
We shall inductively construct a sequence \((D_n)\) of subsets of
\(c_{00}\) and a decomposition of \(\mathbb{N}\) into a
sequence of successive intervals \((\Delta_n)\), so that 
\(D = \cup_n D_n\) induces a norm on \(c_{00}\) and \(X\) is the
completion of \(c_{00}\) under this norm. Moreover, the unit vector basis
\((e_n)\) of \(c_{00}\) will be a normalized shrinking basis for \(X\)
which, along with \((D_n)\) and \((\Delta_n)\), will fulfill the hypotheses
of Theorem \ref{T1}. It will then follow that \(X^*\) is isomorphic to
\(\ell_1\). 

Let 
\(\mathcal{F} = \{ F \subset \mathbb{N}: \, |F| \leq \min F + 2\} \cup \{\emptyset\}\).
This is a hereditary and pointwise compact family of finite subsets of \(\mathbb{N}\).
Note that \(F \in \mathcal{F}\) is a maximal, under inclusion, member of
\(\mathcal{F}\) precisely when \(|F| = \min F + 2\). Fix a scalar
\(0 < b < 1/4\).

We inductively construct a sequence \(\Delta_1 < \Delta_2 < \dots \)
of successive intervals of \(\mathbb{N}\) whose union is \(\mathbb{N}\)
and a sequence \((D_n)\) of subsets of \(c_{00}\) so that for all \(n \in \mathbb{N}\)
the following properties are satisfied:
\begin{enumerate}
\item \(e_i^* \in D_n \) for all \(i \in \Delta_n\).
\item \(\mathrm{ supp } \, d^* \subset \cup_{k=1}^n \Delta_k\)
and \(\max \mathrm{ supp } \, d^* \in \Delta_n \) for all \(d^* \in D_n\).
\item \(|\mathrm{ supp } \, d^* \cap \Delta_k | \leq 1\) for all \(d^* \in D_n\)
and \(k \leq n\).
\item \(d^*(i) \in \{b^k : \, k \in \mathbb{N}\} \cup \{0, 1\}\) for all \(i \in \mathbb{N}\)
and \(d^* \in D_n\).
\item \(\mathrm{ supp } \, d^* \in \mathcal{F}\) for all \(d^* \in D_n\).
\item If \(n \geq 3\) and \(i \in \Delta_n\), then there exists a unique
\(d^* \in D_n\) with \(|\mathrm{ supp } \, d^* | > 1\) and
 \(\max \mathrm{ supp } \, d^* = i\).
\item If \(n \geq 3\) and \(d^* \in D_n\) with \(|\mathrm{ supp } \, d^* | > 1\), then we may write
\[ d^* = d_1^* + b (d_2^* | \cup_{j=k + 1}^l \Delta_j ) + e_i^*\]
for some \(1 \leq k < l \leq n -1\), \(d_1^* \in D_k\), \(d_2^* \in D_l\)
and \(i \in \Delta_n\).
\end{enumerate}
Indeed, define \(\Delta_k =\{k\}\) and \(D_k = \{e_k^*\}\) for \(k \leq 2\).
Assume that \(\Delta_k\) and \(D_k\) have been defined for all \(k \leq n\).
To define \(D_{n+1}\) and \(\Delta_{n+1}\) we need to introduce some terminology.

If \(\xi^*\) and \(\eta^*\) are elements of \(\cup_{k=1}^n D_k\), then we say that
\((\xi^*, \eta^*)\) is a {\em linked pair } provided that there exist integers
\(1 \leq k < l \leq n\) with \(\xi^* \in D_k\), \(\eta^* \in D_l\)
and \(\mathrm{ supp } \, \xi^* \cup \mathrm{ supp } [ \eta^* | (\cup_{j=k+1}^l \Delta_j)]
\in \mathcal{F}\) without being a maximal element.  
Denote by
\(\Sigma_n\) the set of all possible linked pairs formed by elements of \(\cup_{k=1}^n D_k\).
Note that \( (e_i^*, e_j^*) \in \Sigma_n\) whenever \( i \in \Delta_k\),
\(j \in \Delta_l\) and \(1 \leq k < l \leq n\).
Let \(\Delta_{n+1}\) be the interval adjacent to \(\Delta_n\) having \(|\Sigma_n|\) elements. 
Let \(\sigma_n \colon \Sigma_n \to \Delta_{n+1}\) be an injection. Define
\begin{align}
D_{n+1} = \{ \xi^* + b \eta^* | (\cup_{i=k+1}^l \Delta_i) + e_{\sigma_n(\xi^*, \eta^*)}^* \, &: \notag \\
&(\xi^*, \eta^*) \in \Sigma_n\} \cup \{ e_i^* : \, i \in \Delta_{n+1}\} \notag
\end{align}
The inductive construction of \((D_n)\) and \((\Delta_n)\) is now complete.
It is straightforward to check that properties \((1)-(7)\) are satisfied.

Let \(D = \cup_n D_n\) and define a norm on \(c_{00}\) by
\[\|x\| = \sup \{ |\sum_i d^*(i) x(i)| : \, d^* \in D \}\]
This is indeed a norm because of \((1)\). 
\(X\) is the completion of \(c_{00}\) under this norm. Observe that, under natural
identification, \(D \subset B_{X^*}\).
\begin{Lem} \label{L1}
\((e_n)\) is a normalized, shrinking  basis for \(X\) and \(X^*\) is isomorphic to \(\ell_1\).
\end{Lem}
\begin{proof}
Property \((4)\) implies that \(|d^*(i)| \leq 1\) for all
\(i \in \mathbb{N}\) and \(d^* \in D\). Combining this with \((1)\) we 
infer that \((e_n)\) is a normalized sequence in \(X\).

To show that \((e_n)\) is a Schauder basis for \(X\) it suffices proving that
\(\|d^* | I \| \leq 2\)
for all \(n \in \mathbb{N}\), \(d^* \in D_n\) and initial intervals \(I\) of \(\mathbb{N}\).
This is done by induction on \(n\), the cases \(n \leq 2\) being trivial.
Assuming that the assertion holds for some \(n \geq 2\), let \(d^* \in D_{n+1}\).
Property \((7)\) allows us write
\[ d^* = d_1^* + b (d_2^* | \cup_{j=k + 1}^l \Delta_j ) + e_i^*\]
for some \(1 \leq k < l \leq n\), \(d_1^* \in D_k\), \(d_2^* \in D_l\)
and \(i \in \Delta_{n+1}\). Let \(I\) be an initial interval of \(\mathbb{N}\).
If \(i \in I\), then \(d^* | I  = d^*\) and so \(\|d^* | I\| \leq 1\). Next assume
\(i \notin I\). If \(\max I \leq \max \Delta_k\), then \(d^* | I = d_1^* | I\)
and the assertion follows from the induction hypothesis applied on \(d_1^*\).
In case \(\max I > \max \Delta_k\), consider the interval \( J = I \cap \cup_{j=k + 1}^l \Delta_j \).
The induction hypothesis applied on \(d_2^*\) now yields that \(\|d_2^* | J \| \leq 4\).
Thus, 
\[ \|d^* | I \| \leq \|d_1^* \| + b \| d_2^* | J \| \leq 1 + (1/4)4 =2\]
completing the inductive step. Therefore, \((e_n)\) is a normalized Schauder basis
for \(X\).

Evidently, properties \((1)-(7)\) guarantee that \(X\),
\((e_n)\), \(D\) and \((\Delta_n)\) satisfy the hypotheses of Theorem \ref{T1}
and hence \(X\) is an \(\mathcal{L}_\infty\) space. We finally show that
\((e_n)\) is shrinking which in turn will give us that \(X^*\) 
is isomorphic to \(\ell_1\). To this end, it is sufficient to show that
every normalized block basis \((u_n)\) of \((e_n)\) is weakly null. We know, because
of \((5)\), that the supports of the elements of \(D\) belong to the hereditary and compact
family \( \mathcal{F}\). The same holds true for elements in \(\overline{D}^{w^*}\)
and so \(\lim_n d^* u_n  = 0\) for every \(d^* \in \overline{D}^{w^*}\).
We deduce from Rainwater's Theorem that \((u_n)\) is weakly null.
\end{proof}
We remark that because of \((4)\), \(\overline{D}^{w^*}\) is countable
and so \(X\) embeds isometrically into \(C(\omega^{\omega^\alpha})\) for some \(\alpha < \omega_1\)
(\cite{MS}, \cite{BP}).
In the next lemma we describe the smallest such \(\alpha\).
\begin{Lem} \label{L2}
\(X\) is isometric to a subspace of \(C(\omega^\omega)\).  
\end{Lem}
\begin{proof}
We first introduce some notation. Let \(R\) be a countable, compact subset
of non-negative reals having \(0\) as a unique cluster point. Given a hereditary and
pointwise compact family \(\mathcal{G}\) of finite subsets of \(\mathbb{N}\), we set
\[ R * \mathcal{G} = \{f \colon \mathbb{N} \to R : \, \mathrm{ supp } \, f \in \mathcal{G} \} 
\cup \{\vec{0} \}\]
where \(\vec{0}\) stands for the zero function.
Endowing \( K = R^{\mathbb{N}}\) with the cartesian topology, we see that
\( R * \mathcal{G}\) is a countable, compact subset of \(K\).

Given \(n \in \mathbb{N}\), let \(\mathcal{A}_n = \{F \subset \mathbb{N} : \, |F| \leq n\}\).
Also let \(\mathcal{A}_\omega = \mathcal{F}\). It is well known that for all ordinals \(\xi \leq \omega\),
\(\mathcal{A}_\xi \)
is a hereditary and pointwise compact family of finite subsets of \(\mathbb{N}\) 
such that \({\mathcal{A}_\xi}^{(\xi)} = \{\emptyset\}\).

{\em Claim}: \((R * \mathcal{A}_\xi)^{(\xi)} = \{\vec{0}\}\)
for all \(\xi \leq \omega\).

We first prove the claim when \(\xi = n < \omega\), by induction on \(n \in \mathbb{N}\).
When \(n =1\), the proof is straightforward since 
\[R * \mathcal{A}_1 = \{ r \chi_{\{k\}} : \, r \in R, \, k \in \mathbb{N} \}\] 
Assume the assertion is true for some \(n \in \mathbb{N}\). 
Notice that for each \(k \leq n\) we have that 
\[R * \mathcal{A}_k = \{ f \in R * \mathcal{A}_n : \, |\mathrm{ supp } \, f | \leq k \} \cup \{\vec{ 0}\}\]
It is also easy to see that the set of the isolated points of \(R * \mathcal{A}_n\) consists precisely of
functions in \(R * \mathcal{A}_n\) whose support contains exactly \(n\) points. It follows now that
\[ (R * \mathcal{A}_n)^{(1)} =  R * \mathcal{A}_{n-1}\]
and hence \((R * \mathcal{A}_n)^{(n)} = \{\vec{0}\}\) because of the induction hypothesis.
(Here we made use of the following easy fact: Suppose \(A\) is a countable compact metric space, \(B \)
a closed subset of \(A\) with \(B^{(m)} = \{t\}\) for some \(m \in \mathbb{N}\) and
\(t \in B\). If \(A \setminus B\) is the set of the isolated points of \(A\), then
\(A^{(m + 1)} = \{t\}\)).

We next prove the claim for \(\xi = \omega\). For every \(n \in \mathbb{N}\) set
\[ L_n = \{ f \in R * \mathcal{A}_\omega : \, |\mathrm{ supp } \, f | \leq n \} \cup \{\vec{ 0}\}\]
Of course, \(L_n\) is a closed subset of \(R * \mathcal{A}_\omega\) 
and \(L_n \subset R * \mathcal{A}_n\) for all \(n \in \mathbb{N}\).
It is also clear that \(R * \mathcal{A}_\omega = \cup_{n=1}^\infty L_n\).
We now observe that if \((f_n)\) is a sequence in \(R * \mathcal{A}_\omega\) converging
to \(f \in R * \mathcal{A}_\omega\) and there exist integers \(k_1 < k_2 < \dots\) with
\(f_n \in L_{k_{n+1}} \setminus L_{k_n} \), for all \(n \in \mathbb{N}\), then
\(f = \vec{0}\).
Indeed, if \(d_n = \min \mathrm{ supp } \, f_n \), then \(d_n > k_n -2\) as
\(\mathrm{ supp } \, f_n \in \mathcal{F}\) for all \(n \in \mathbb{N}\).
It follows that \(\lim_n d_n = \infty\) and so \(f = \vec{0}\).

We conclude from the above that if \(L\) is a closed subset of 
\(R * \mathcal{A}_\omega\) and \(\vec{0} \notin L\), then
\(L \subset L_n\) for some \(n \in \mathbb{N}\). Since the claim holds when \(\xi = n\)
we obtain that \(L^{(n+1)} = \emptyset \) . We deduce from this
that \((R * \mathcal{A}_\omega)^{(\omega)} \subset  \{\vec{0}\}\).
For the reverse inclusion, fix some \(r > 0\) in \(R\) and notice that for all \(n \in \mathbb{N}\)
\[M_n = \{ f \in R * \mathcal{A}_\omega : \, f(n) = r, \, n = \min \mathrm{ supp } \, f \}\]
is homeomorphic to \(R * \mathcal{A}_{n + 1}\) with \(M_n^{(n+1)} = \{r \chi_{\{n\}} \}\). 
Therefore, \( r \chi_{\{n\}} \in (R * \mathcal{A}_\omega)^{(n+1)}\) for all 
\(n \in \mathbb{N}\) and thus \( \vec{0} \in (R * \mathcal{A}_\omega)^{(\omega)} \).
This completes the proof of the claim.

We are now able to complete the proof of the lemma. To this end, let
\(R = \{ b^n : n \in \mathbb{N}\} \cup \{0,1\}\). We define a map
\(\phi \colon \overline{D}^{w^*} \to R * \mathcal{F} \) 
by the rule 
\[\phi d^* = \sum_{ n \in \mathrm{ supp } \, d^*} d^*(n) \chi_{\{n\}}\]
for all \(d^* \in \overline{D}^{w^*}\) (\( \phi 0 = \vec{0}\)).
Properties  \((4)\) and \((5)\) of \(D\) yield that \(\phi\) is well defined, injective and
continuous when \(\overline{D}^{w^*}\) is endowed the \(w^*\)-topology. It follows now
that \(\overline{D}^{w^*}\) is \(w^*\)-homeomorphic to a subset of \([1, \omega^\omega]\)
and so \(X\) is isometric to a subspace of \(C(\omega^\omega)\).
\end{proof}
In the sequel, if \(u = \sum_{n=1}^\infty a_n e_n \) is a vector in \(X\), we set
\(\|u\|_{c_0} = \sup_n |a_n|\).
\begin{Lem} \label{L3}
Let \((u_n)\) be a normalized block basis of \((e_n)\). Then the following
properties hold:
\begin{enumerate}
\item If \(\lim_n \|u_n\|_{c_0} = 0\), then 
some subsequence of \((u_n)\) is equivalent to the \(c_0\) basis.
\item If there is some \(\delta > 0\) so that \(\|u_n\|_{c_0} > \delta\)
for all \(n \in \mathbb{N}\), then some subsequence of \((u_n)\)
admits \(\ell_1\) as a spreading model.
\end{enumerate}
\end{Lem}
\begin{proof}
To prove \((1)\) let \(k_n = \max \mathrm{ supp } \, u_n \) for all \(n \in \mathbb{N}\).
Passing to a subsequence of \((u_n)\) if necessary, using the fact that \(\lim_n \|u_n\|_{c_0} = 0\),
we may assume without loss of generality, that 
\[ \|u_n \|_{c_0} < 1/2^{k_{n-1}}, \, \forall \, n \geq 2\]
If \(d^* \in D\), let \(m\) denote the smallest integer in the support of \(d^*\).
Note that \(| \mathrm{ supp } \, d^* | \leq m+2\), thanks to \((5)\). Next, let \(n_1\)
be the smallest integer \(n\) for which \( \mathrm{ supp } \, u_n\) intersects
\(\mathrm{ supp } \, d^* \). It follows that \(m \leq k_{n_1}\) and therefore
\begin{align}
\sum_{n=1}^\infty |d^*u_n | &\leq |d^* u_{n_1} | + \sum_{n= n_1 + 1 }^\infty |d^*u_n | 
\leq 1 + \sum_{n= n_1 + 1 }^\infty (m+2)/2^{k_{n-1}} \notag \\ 
&\leq 1 + \sum_{n= n_1 + 1 }^\infty (k_{n-1}+2)/2^{k_{n-1}} 
\leq 1 + \sum_{n=1}^\infty (n+2)/2^n \notag
\end{align}
We conclude that \((u_n)\) is equivalent to the \(c_0\) basis.
For the second part of this lemma, let \(i_n \in \mathrm{ supp } \, u_n \)
satisfy \( |e_{i_n}^* u_n | \geq \delta\) for all \(n \in \mathbb{N}\).

{\em Claim}: For every \(M \in [\mathbb{N}]\) and every \(k \in \mathbb{N}\)
there exist \(F \subset M\) with \(|F| = k\) and \(d^* \in D\)
with \(\min M \leq \min \mathrm{ supp } \, d^* \) and \(| \mathrm{ supp } \, d^* | \leq 2k +1 \), so that
\(\mathrm{ supp } \, d^* \cap \mathrm{ supp } \, u_n = \{i_n\}\) and
\(|d^* u_n | \geq \delta b \) 
for all \(n \in F\).

We prove the claim by induction on \(k\). The case \(k=1\) is trivial
as \(e_j^* \in D\) for all \(j \in \mathbb{N}\).
Assume that \(k \geq 2\) and that the claim holds for \(k-1\). 
Without loss of generality assume that \(\min M > 2k +1 \).
The induction hypothesis yields some \(d_1^* \in D\) and \(F_1 \subset M\)
with \(\min M \leq \min \mathrm{ supp } \, d_1^* \), \(| \mathrm{ supp } \, d_1^* | \leq 2k -1 \),
\(|F_1| = k-1\), \(\mathrm{ supp } \, d_1^* \cap \mathrm{ supp } \, u_n = \{i_n\}\)
and \(|d_1^* u_n | \geq \delta b \)  
for all \(n \in F_1\). Let \(m = \max F_1\). Choose \(l \in \mathbb{N}\) so that both
\(\mathrm{ supp } \, d_1^*\) and \(\mathrm{ supp } \, u_m \) are contained in
\(\cup_{j=1}^l \Delta_j \). We next choose \( m_1 \in M\), \(m_1 > m\), so that
\(\min \mathrm{ supp } \, u_{m_1} > \max \Delta_l\). It follows now that
\((d_1^* , e_{i_{m_1}}^*)\) is a linked pair. Choose \(l_1 \in \mathbb{N}\)
so that \(\min \Delta_{l_1} > \max \mathrm{ supp } \, u_{m_1}\). We may now find
some \(t \in \Delta_{l_1}\) so that
\[ d^* = d_1^* + b e_{i_{m_1}}^* + e_t^* \in D\]
Evidently, \(d^*\) and \(F = F_1 \cup \{m_1\}\) satisfy the claim for \(k\).

We deduce now from the preceding claim that for every \(k \in \mathbb{N}\) and
every subsequence \((u_{m_n})\) of \((u_n)\), there exist \(d^* \in D\)
and indices \(m_{i_1} < \dots < m_{i_k}\) so that \(|d^* u_{m_{i_j}} | \geq \delta b\)
for all \(j \leq k\). An application of Ramsey's theorem and a diagonalization argument 
finally yield a subsequence \((u_{k_n})\) of \((u_n)\)
so that for every \(F \subset \mathbb{N}\) with \(|F| \leq \min F\) there exists \(d^* \in D\)
such that \(|d^* u_{k_n} | \geq \delta b\) for all \(n \in F\). Therefore, some subsequence of
\((u_{k_n})\) admits \(\ell_1\) as a spreading model.
\end{proof}
\begin{proof}[Proof of Theorem \ref{mT}]
The first part of the theorem follows directly from Lemmas \ref{L1} and \ref{L2}.
For the second part, let \((x_n)\) be a normalized weakly null sequence in \(X\).
By passing to a subsequence if necessary, we may assume without loss of generality,
that \((x_n)\) is equivalent to a normalized block basis \((u_n)\) of \((e_n)\).
Then either \(\lim_n \|u_n \|_{c_0} = 0\), in which case \((1)\) of Lemma \ref{L3}
implies that some subsequence of \((x_n)\) is equivalent to the \(c_0\) basis,
or, there is some \(\delta > 0\) and a subsequence \((u_{k_n})\) of \((u_n)\)
so that \(\|u_{k_n} \|_{c_0} \geq \delta\) for all \(n \in \mathbb{N}\).
It follows now by \((2)\) of Lemma \ref{L3}, that some subsequence of \((x_n)\)
admits \(\ell_1\) as a spreading model.
It is shown in \cite{gow}, that every normalized weakly null sequence in \(C(\omega^\omega)\)
which admits \(\ell_1\) as a spreading model, has a subsequence equivalent to a
subsequence of the natural basis of Schreier's space. Since \(X\) is isometric to
a subspace of \(C(\omega^\omega)\), we conclude that \((x_n)\) has a subsequence equivalent to
a subsequence of the natural basis of Schreier's space. 
We finally note that \((e_n)\) clearly satisfies \((2)\) of Lemma \ref{L3} and so it has 
a subsequence equivalent to
a subsequence of the natural basis of Schreier's space. 
\end{proof}

\end{document}